\documentclass{article}

\usepackage{amssymb,amsmath,mathrsfs,mathtools,amsthm}
\usepackage[a4paper,includeheadfoot,left=30mm,right=30mm,top=25mm,bottom=30mm]{geometry} %sets up the margins
\usepackage{verbatim}

\newtheorem{thm}{Theorem}[section]

\newtheorem{prop}[thm]{Proposition}

\newtheorem{defn}[thm]{Definition}

\theoremstyle{definition}

\theoremstyle{remark}

\newcommand{\hsforall}{\hspace{1mm}\forall\hspace{1mm}}						 %Spacing around \forall
\renewcommand{\Re}{\operatorname*{Re}}                             %Real part of something
\renewcommand{\Im}{\operatorname*{Im}}                             %Imaginary part of something
\renewcommand{\d}{\ensuremath{\,\mathrm{d}}}							         %an upright d for infintessimals
\renewcommand{\geq}{\geqslant}                                     %Always use \geqslant, never the standard \geq
\renewcommand{\leq}{\leqslant}                                     %Always use \leqslant, never the standard \leq
\newcommand{\BE}{\begin{equation}}                                 %Begin an equation environment
\newcommand{\EE}{\end{equation}}                                   %End an equation environment
\newcommand{\BES}{\begin{equation*}}                               %Begin an equation* environment
\newcommand{\EES}{\end{equation*}}                                 %End an equation* environment
\newcommand{\BP}{\begin{pmatrix}}                                  %Begin a pmatrix environment
\newcommand{\EP}{\end{pmatrix}}                                    %End a pmatrix environment
\newcommand{\R}{\mathbb{R}}                                        %Shorthand for the set of real numbers
\newcommand{\C}{\mathbb{C}}                                        %Shorthand for the set of complex numbers
\def\clap#1{\hbox to 0pt{\hss#1\hss}}

\numberwithin{equation}{section}

\title{The unified transform method for linear initial-boundary value problems: a spectral interpretation}
\author{D. A. Smith\\
\footnotesize Department of Mathematical Sciences, University of Cincinnati, OH \\
\footnotesize email\textup{: \texttt{david.smith2@uc.edu}}
}

\begin{document}
\maketitle

\abstract{It is known that the unified transform method may be used to solve any well-posed initial-boundary value problem for a linear constant-coefficient evolution equation on the finite interval or the half-line. In contrast, classical methods such as Fourier series and transform techniques may only be used to solve certain problems. The solution representation obtained by such a classical method is known to be an expansion in the eigenfunctions or generalised eigenfunctions of the self-adjoint ordinary differential operator associated with the spatial part of the initial-boundary value problem. In this work, we emphasise that the unified transform method may be viewed as the natural extension of Fourier transform techniques for non-self-adjoint operators. Moreover, we investigate the spectral meaning of the transform pair used in the new method; we discuss the recent definition of a new class of spectral functionals and show how it permits the diagonalisation of certain non-self-adjoint spatial differential operators.}

% Section 1
\section{Introduction} \label{sec:Intro}

Consider the Dirichlet problem for the heat equation on a finite interval,
\begin{align*}
(\partial_t-\partial_{xx})q(x,t) &= 0, & (x,t) &\in (0,1)\times(0,T) \\
q(x,0) &= q_0(x), & x &\in [0,1] \\
q(0,t) &= q(1,t) = 0, & t &\in [0,T],
\end{align*}
where $q_0$ is some known initial datum, compatible with the boundary conditions. In solving this problem by separation of variables, one begins with the assumption that $q(x,t)=\xi(x)\tau(t)$ but rapidly discovers that this ansatz is incompatible with general initial data, so replaces it with the more general assumption that $q$ can be expressed as a convergent series of such functions $\xi$, with corresponding coefficients $\tau$. Of course, these $\xi$ are precisely the eigenfunctions of a certain self-adjoint differential operator,
\begin{align*}
S &: \{f\in C^\infty[0,1]:f(0)=f(1)=0\}\to C^\infty[0,1] & Sf=-\frac{\d^2 f}{\d x^2},
\end{align*}
so it is no surprise that it is possible to write the solution in such a form. Crucially, it is the general results on completeness of eigenfunctions and convergence of eigenfunction series expansions for such operators that ensure the series representation is valid.

However, it is easy to construct a similar problem, for which the corresponding differential operator is non-self-adjoint. Consider \\
\noindent{\bf Problem~1}
\begin{subequations} \label{eqn:problem.LKdV.Uncoupled.FI}
\begin{align}
(\partial_t+\partial_{xxx})q(x,t) &= 0, & (x,t) &\in (0,1)\times(0,T) \\
q(x,0) &= q_0(x), & x &\in [0,1] \\
q(0,t) &= q(1,t) = q_x(1,t) = 0, & t &\in [0,T],
\end{align}
\end{subequations}
for which the corresponding operator
\begin{align} \label{eqn:S.FI}
S &: \{f\in C^\infty[0,1]:f(0)=f(1)=f'(1)=0\}\to C^\infty[0,1] & Sf=-i\frac{\d^3 f}{\d x^3}
\end{align}
is formally self-adjoint, but is non-self-adjoint because of its domain. In this case, one cannot rely on the classical theorems for the convergence of the eigenfunction expansions in the space of initial data.

Indeed, it turns out that while the (generalised) eigenfunctions do form a complete system~\cite{Shk1976a}, the corresponding expansions do not converge~\cite{Jac1915a} except in very special cases~\cite{Hop1919a}, so an ansatz such as the classical separation of variables series representation cannot be valid. Nevertheless, the unified transform method can be used to obtain a solution to problem~\eqref{eqn:problem.LKdV.Uncoupled.FI}. In this work, we describe how the solution representation thence derived can be interpreted as an expansion in certain spectral objects, which are defined as an abstraction of eigenfunctions and the generalized eigenfunctions of Gel'fand.

\subsection{Initial-boundary value problems}

In this work, we concentrate on two examples of non-self-adjoint initial-boundary value problems. We study problem~1, described above, and also the following \\
\noindent{\bf Problem~2}
\begin{subequations} \label{eqn:problem.LKdV.Uncoupled.HL}
\begin{align}
(\partial_t+\partial_{xxx})q(x,t) &= 0, & (x,t) &\in (0,\infty)\times(0,T) \\
q(x,0) &= q_0(x), & x &\in (0,\infty) \\
q(0,t) &= 0, & t &\in (0,T).
\end{align}
\end{subequations}
The corresponding spatial differential operator is
\begin{align} \label{eqn:S.HL}
S &: \{f\in \mathcal{E}[0,\infty):f(0)=0\}\to \mathcal{S}[0,\infty) & Sf=-i\frac{\d^3 f}{\d x^3},
\end{align}
where for some arbitrarily small $\epsilon>0$
\BE
\mathcal{E}[0,\infty) = \{ f\in\mathcal{S}[0,\infty):f(x)=O(e^{-\epsilon x})\mbox{ as } x\to\infty \}
\EE
and $\mathcal{S}[0,\infty)$ is the Schwartz class of rapidly-decaying smooth functions restricted to the half-line. Our reasons for working in $\mathcal{E}[0,\infty)$ rather than $\mathcal{S}[0,\infty)$ are discussed after the proof of theorem~\ref{thm:Flambda.are.AugEig}.

Problem~1 is posed on the finite interval, and is studied in detail in~\cite{FS2013a} whereas problem~2 is posed on the half-line, the case covered by~\cite{PS2014a}. Here we study the two problems together, to draw attention to the similarities between the finite interval and half-line frameworks.

\medskip

It should be noted that problems such as these are open to analysis via a Laplace transform in time. However, as described in~\cite[appendix~C]{FP2005a}, this method presents a number of difficulties, which do not occur with the unified transform method, particularly when one studies problems of still higher order or considers certain inhomogeneous boundary conditions. The surveys~\cite{DTV2014a,FS2012a} give good accounts of the differences between the unified transform method and classical methods.

The purpose of this work is to give a spectral interpretation of the success of the unified transform method. In order to motivate and inform the discussion, we shall refer to the familiar spectral understanding of the classical Fourier methods, which comprise an analysis of the spatial part of the partial differential equation. The Laplace transform in the temporal part of the equation is not so clearly analogous to the unified method and is not discussed hereafter.

\subsection{Layout of paper}

In section~\ref{sec:UTM}, we give the solution of problems~1 and~2, as obtained through the unified transform method. In section~\ref{ssec:Transforms}, we reformulate these solution representations as a transform-inverse transform pair, whereupon the unified transform method may be seen as the method for deriving said transform pair, tailor-made for the particular IBVP of interest.

In section~\ref{sec:Spectral}, we consider the spectral meaning of the transform pair. Motivated by a discussion of Gel'fand's generalised eigenfunctions as applied to the IBVP for the self-adjoint Dirichlet heat operator, we define a new class of spectral functionals, whose properties are more suited to non-self-adjoint operators such as~\eqref{eqn:S.FI} and~\eqref{eqn:S.HL}. We discuss the characteristics of the new definition and show how the transform pair provides a new spectral diagonalisation result for these operators.

\section{Unified transform method} \label{sec:UTM}

The unified transform method~\cite{Fok2008a}, as applied to half-line initial-boundary value problems for linear evolution equations, was first described in~\cite{Fok1997a} and formalised in~\cite{FS1999a}. The method was extended to problems on the finite interval in~\cite{FP2001a} and detailed well-posedness results were established first for simple, uncoupled boundary conditions and then in full generality~\cite{Pel2004a,Smi2012a}.

In particular,~\cite{FS1999a} established rigourously that the solution of a linear initial-boundary value problem may be obtained without any assumptions beyond well-posedness of the problem. This is in contrast to classical methods, where stronger assumptions on the form of the solution are imposed. Indeed, as discussed above, the d'Alembert method of separation of variables requires that the eigenfunctions of the spatial differential operator form a complete system in the space of admissible initial data and that the expansion of the initial datum in these eigenfunctions converges. In applying the method of Fourier transforms, one implicitly assumes that the generalised eigenfunctions form a complete system in the space of initial data and that the corresponding expansion converges. This is certainly a weaker ansatz than completeness and convergence for ordinary eigenfunctions, and it is valid whenever the operator happens to be self-adjoint~\cite{GV1964a}, but it is not guaranteed for non-self-adjoint operators; formal self-adjointness is insufficient.

When the unified transform method is applied to problem~2, the solution is obtained in the form of a contour integral,
\BE \label{eqn:introIBVP.solution.HL}
q(x,t) = \frac{1}{2\pi}\int_{\Gamma^+} e^{i\lambda x + i\lambda^3t} \zeta^+(\lambda;q_0)\d\lambda + \frac{1}{2\pi}\int_{\Gamma^-} e^{i\lambda x + i\lambda^3t} \zeta^-(\lambda;q_0)\d\lambda,
\EE
where
\begin{align*}
\Gamma^+ &\mbox{ is the boundary of the domain } \{\lambda\in\C^+:\Re(-i\lambda^3)<0 \mbox{ and } |\lambda|>1\} \mbox{ as shown on figure~\ref{fig:HL-cont},} \\
\Gamma^- &\mbox{ is } \R \mbox{ perturbed away from } 0 \mbox{ along a small (radius $<\epsilon$) semicircular arc in } \C^+.
\end{align*}
for $\lambda\in\C$,
\begin{subequations} \label{eqn:introIBVP.Zeta.HL}
\begin{align} \label{eqn:introIBVP.Zetaplus.HL}
\zeta^+(\lambda;q_0) &= \alpha\hat{q}_0(\alpha\lambda) + \alpha^2\hat{q}_0(\alpha^2\lambda), \\ \label{eqn:introIBVP.Zetaminus.HL}
\zeta^-(\lambda;q_0) &= \hat{q}_0(\lambda),
\end{align}
\end{subequations}
$\alpha$ is the cube root of unity $e^{2\pi i/3}$ and $\hat{q}_0$ is the Fourier transform
\BE
\hat{q}_0(\lambda) = \int_0^\infty e^{-i\lambda x}q_0(x)\d x, \qquad \lambda\in\C: \Im(\lambda)<\epsilon.
\EE

\begin{figure}
\begin{center}
\includegraphics{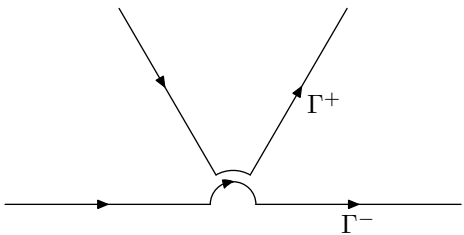}
\caption{Contours for problem~2.}
\label{fig:HL-cont}
\end{center}
\end{figure}

Applying the unified transform method to problem~1, one obtains the following contour integral representation of the solution:
\BE \label{eqn:introIBVP.solution.FI}
q(x,t) = \frac{1}{2\pi}\int_{\Gamma^+} e^{i\lambda x + i\lambda^3t} \frac{\zeta^+(\lambda;q_0)}{\Delta(\lambda)}\d\lambda + \frac{1}{2\pi}\int_{\Gamma^-} e^{i\lambda(x-1) + i\lambda^3t} \frac{\zeta^-(\lambda;q_0)}{\Delta(\lambda)}\d\lambda,
\EE
where
\begin{align*}
\Gamma^\pm &\mbox{ are the boundaries of the sectors } \{\lambda\in\C^\pm:\Im(\lambda^3)>0 \mbox{ and } |\lambda|>1\} \mbox{ as shown on figure~\ref{fig:FI-cont},}
\end{align*}
for $\lambda\in\C$,
\begin{subequations} \label{eqn:introIBVP.DeltaZeta.FI}
\begin{align}
\Delta(\lambda) &= e^{-i\lambda} + \alpha e^{-i\alpha\lambda} + \alpha^2 e^{-i\alpha^2\lambda}, \\
\zeta^+(\lambda;q_0) &= \hat{q}_0(\lambda)(\alpha e^{-i\alpha\lambda}+\alpha^2 e^{-i\alpha^2\lambda}) - (\alpha\hat{q}_0(\alpha\lambda) + \alpha^2\hat{q}_0(\alpha^2\lambda))e^{-i\lambda}, \\
\zeta^-(\lambda;q_0) &= -\hat{q}_0(\lambda) - \alpha\hat{q}_0(\alpha\lambda) - \alpha^2\hat{q}_0(\alpha^2\lambda).
\end{align}
\end{subequations}
$\alpha$ is the root of unity $e^{2\pi i/3}$ and $\hat{q}_0(\lambda)$ is the Fourier transform
\BE
\hat{q}_0(\lambda) = \int_0^1 e^{-i\lambda x}q_0(x)\d x, \qquad \lambda\in\C.
\EE

\begin{figure}
\begin{center}
\includegraphics{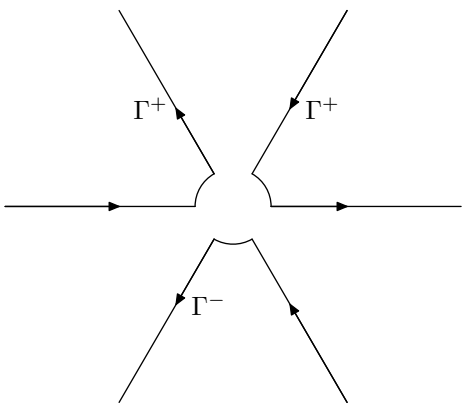}
\caption{Contours for problem~1.}
\label{fig:FI-cont}
\end{center}
\end{figure}

The solution of each problem involves contour integrals of sums of Fourier transforms of the initial data. This is in contrast to the classical Fourier series expansion one might expect from studying the finite interval Dirichlet heat problem. Indeed, many finite-interval IBVP for higher-order PDE have solutions which may be expressed as a series of (generalised) eigenfunctions in a similar manner. Having solved such a problem using the unified transform method, obtaining a solution representation similar to equation~\eqref{eqn:introIBVP.solution.FI}, one may use Jordan's lemma to perform a contour deformation, closing each component of $\Gamma^\pm$ onto the poles of the integrand. A residue calculation then yields a series representation of the solution, which corresponds to the classical eigenfunction expansion~\cite{Chi2006a,Smi2012a}. But for some odd order problems, including problem~1, this contour deformation and series representation is impossible~\cite{Pel2005a}.

\subsection{Transform-inverse transform pair} \label{ssec:Transforms}

Some problems similar to those studied here can be solved using the method of Fourier transforms. The classical Fourier sine and Fourier cosine transforms are typically used to solve the heat equation with Dirichlet and Neumann boundary conditions respectively. Indeed a different transform (and corresponding inverse transform) is required for each IBVP.

One may view the unified transform method as a way of determining an applicable transform pair for a given IBVP; we may rewrite the solution representations~\eqref{eqn:introIBVP.solution.HL} and~\eqref{eqn:introIBVP.solution.FI} in terms of Fourier-type integral transforms as
\BE
q(x,t) = f_x\left(e^{i\lambda^3t}F_\lambda(q_0)\right),
\EE
where
\begin{subequations} \label{eqn:introTrans.1.1}
\begin{align} \label{eqn:introTrans.1.1a}
f(x) &\mapsto F(\lambda): & F_\lambda(f) &= \begin{cases} \int_0^1 \phi^+(x,\lambda)f(x)\d x & \mbox{if } \lambda\in\Gamma^+, \\ \int_0^1 \phi^-(x,\lambda)f(x)\d x & \mbox{if } \lambda\in\Gamma^-, \end{cases} \\ \label{eqn:introTrans.1.1b}
F(\lambda) &\mapsto f(x): & f_x(F) &= \left\{ \int_{\Gamma^+} + \int_{\Gamma^-} \right\} e^{i\lambda x} F(\lambda) \d\lambda, \qquad x\in[0,1], \mbox{ or } x\in[0,\infty),
\end{align}
the contours $\Gamma^\pm$ are the appropriate contours for the problem, as defined above, and the kernels $\phi^\pm$ are given by
\begin{align} \label{eqn:introTrans.1.1c}
\phi^+(x,\lambda)   &= \frac{1}{2\pi\Delta(\lambda)} \left[ e^{-i\lambda x}(\alpha e^{-i\alpha\lambda}+\alpha^2 e^{-i\alpha^2\lambda}) - (\alpha e^{-i\alpha\lambda x} + \alpha^2 e^{-i\alpha^2\lambda x})e^{-i\lambda} \right], \\ \label{eqn:introTrans.1.1d}
\phi^-(x,\lambda)   &= \frac{-e^{-i\lambda}}{2\pi\Delta(\lambda)} \left[ e^{-i\lambda x} + \alpha e^{-i\alpha\lambda x} + \alpha^2 e^{-i\alpha^2\lambda x} \right]
\end{align}
for problem~1 and
\begin{align} \label{eqn:introTrans.1.1e}
\phi^+(x,\lambda)   &= \frac{1}{2\pi} \left[ \alpha e^{-i\alpha\lambda x} + \alpha^2 e^{-i\alpha^2\lambda x} \right], \\ \label{eqn:introTrans.1.1f}
\phi^-(x,\lambda)   &= \frac{1}{2\pi} e^{-i\lambda x}
\end{align}
\end{subequations}
for problem~2.

Although the above is easy to see formally, it must be established that the maps $f(x) \mapsto F(\lambda)$ and $F(\lambda) \mapsto f(x)$ are truly a transform-inverse transform pair.

\begin{prop} \label{prop:TransPairValid}
\begin{enumerate}
  \item[(i)]{Suppose $F_\lambda$ and $f_x$ are defined by equations~\eqref{eqn:introTrans.1.1a}--\eqref{eqn:introTrans.1.1d}.
	  Then for all $f\in C^\infty[0,1]$ with $f(0)=f(1)=f_x(1)=0$ and for all $x\in[0,1]$,
		$f_x(F_\lambda(f))=f(x)$.
		}
  \item[(ii)]{Suppose $F_\lambda$ and $f_x$ are defined by equations~\eqref{eqn:introTrans.1.1a},~\eqref{eqn:introTrans.1.1b},~\eqref{eqn:introTrans.1.1e} and~\eqref{eqn:introTrans.1.1f}.
	  Then for all $f\in \mathcal{E}[0,\infty)$ with $f(0)=0$ and for all $x\in[0,\infty)$,
		$f_x(F_\lambda(f))=f(x)$.
		}
\end{enumerate}
\end{prop}

\begin{proof}~\\
\noindent(i)\hspace{1em}The definition of the transform pair implies
\BE \label{eqn:Transforms.valid:LKdV:prop2:proof.FI1}
f_x(F_\lambda(f)) = \frac{1}{2\pi}\left\{\int_{\Gamma^+} + \int_{\Gamma_0}\right\} e^{i\lambda x} \frac{\zeta^+(\lambda)}{\Delta(\lambda)}\d\lambda + \frac{1}{2\pi}\int_{\Gamma^-} e^{i\lambda(x-1)} \frac{\zeta^-(\lambda)}{\Delta(\lambda)}\d\lambda,
\EE
where $\zeta^\pm$ and $\Delta$ are given by equations~\eqref{eqn:introIBVP.DeltaZeta.FI} and the contours $\Gamma^+$ and $\Gamma^-$ are shown in figure~\ref{fig:FI-cont}.

\begin{figure}
\begin{center}
\includegraphics{FI-contours-03}
\caption{Contour deformation for problem~1 (finite interval).}
\label{fig:FI-contdef}
\end{center}
\end{figure}

The fastest-growing exponentials of the form $e^{\pm i\alpha^j\lambda}$ in the sectors exterior to $\Gamma^\pm$ are indicated on figure~\ref{fig:FI-contdef}a. Each of these exponentials occurs in $\Delta$ and integration by parts shows that the fastest-growing-terms in $\zeta^\pm$ are the exponentials shown on figure~\ref{fig:FI-contdef}a multiplied by $\lambda^{-2}$. Hence the ratio $\zeta^+(\lambda)/\Delta(\lambda)$ decays for large $\lambda$ within the sector $\pi/3\leq\arg\lambda\leq2\pi/3$ and the ratio $\zeta^-(\lambda)/\Delta(\lambda)$ decays for large $\lambda$ within the sectors $-\pi\leq\arg\lambda\leq-2\pi/3$, $-\pi/3\leq\arg\lambda\leq0$. The relevant integrands are meromorphic functions with poles only at the zeros of $\Delta$. The distribution theory of zeros of exponential polynomials~\cite{Lan1931a} implies that the only poles occur within the sets bounded by $\Gamma^\pm$.

The above observations and Jordan's lemma allow us to deform the relevant contours to the contour $\gamma$ shown on figure~\ref{fig:FI-contdef}b; the red arrows on figure~\ref{fig:FI-contdef}a indicate the deformation direction. Hence equation~\eqref{eqn:Transforms.valid:LKdV:prop2:proof.FI1} simplifies to
\BE \label{eqn:Transforms.valid:LKdV:prop2:proof.FI2}
f_x(F_\lambda(f)) = \frac{1}{2\pi}\int_\gamma \frac{e^{i\lambda x}}{\Delta(\lambda)}\left( \zeta^+(\lambda) - e^{-i\lambda}\zeta^-(\lambda) \right) \d\lambda.
\EE
Equations~\eqref{eqn:introIBVP.DeltaZeta.FI} imply,
\BE \label{eqn:Transforms.valid:LKdV:prop2:proof.FI3}
\left( \zeta^+(\lambda) - e^{-i\lambda}\zeta^-(\lambda) \right) = \hat{f}(\lambda)\Delta(\lambda),
\EE
where $\hat{f}$ is the Fourier transform of a piecewise smooth function supported on $[0,1]$. Hence the integrand on the right hand side of equation~\eqref{eqn:Transforms.valid:LKdV:prop2:proof.FI2} is an entire function, so we can deform the contour onto the real axis. The usual Fourier inversion theorem completes the proof.
\smallskip

\noindent(ii)\hspace{1em}The definition of the transform pair implies
\BE \label{eqn:Transforms.valid:LKdV:prop2:proof.HL1}
f_x(F_\lambda(f)) = \frac{1}{2\pi}\int_{\Gamma^+} e^{i\lambda x} \zeta^+(\lambda) \d\lambda + \frac{1}{2\pi}\int_{\Gamma^-} e^{i\lambda x} \hat{f}(\lambda) \d\lambda,
\EE
where $\zeta$ is given by equation~\eqref{eqn:introIBVP.Zetaplus.HL} and the contours $\Gamma^+$ and $\Gamma^-$ are shown in figure~\ref{fig:HL-cont}.

As $\lambda\to\infty$ from within the closed sector $\{\lambda:\frac{\pi}{3}\leq\arg(\lambda)\frac{2\pi}{3}\}$, the exponentials $e^{-i\alpha\lambda}$ and $e^{-i\alpha^2\lambda}$ are bounded. Moreover, integration by parts and the boundary condition $f(0)=0$ yields $\hat{f}(\alpha\lambda)$, $\hat{f}(\alpha^2\lambda) = O(\lambda^{-2})$. These Fourier transforms are also holomorphic in the same sector hence, by Jordan's lemma, the integral over $\Gamma^+$ vanishes. The integrand $e^{i\lambda x}\hat{f}(\lambda)$ is analytic for $\lambda$ with $\Im(\lambda)<\epsilon$ hence admits a deformation of the contour $\Gamma^-$ onto $\R$. The validity of the usual Fourier transform on $\mathcal{S}[0,\infty)$ completes the proof.
\end{proof}

\begin{thm} \label{thm:TransPairSolvesIBVP}
The transform pair $f(x) \mapsto F(\lambda)$ and $F(\lambda) \mapsto f(x)$ as defined above is the `correct' transform pair for the IBVP, in the sense that
\BE
q(x,t) = f_x\left(e^{i\lambda^3t}F_\lambda(q_0)\right).
\EE
\end{thm}

Theorem~\ref{thm:TransPairSolvesIBVP} follows trivially from equations~\eqref{eqn:introIBVP.solution.HL} and~\eqref{eqn:introIBVP.solution.FI} by changing notation to that of section~\ref{ssec:Transforms}. In view of proposition~\ref{prop:TransPairValid}, this theorem is simply a statement that one may solve these problems using the unified transform method, as has been established in~\cite{Fok1997a,FP2001a}. A direct proof of this theorem is given in~\cite{FS2013a,PS2014a} for problems~1 and~2, respectively.

\section{Spectral interpretation} \label{sec:Spectral}

The concept of a generalised eigenfunction, the completeness of and the convergence of expansions in these spectral objects is well-known for self-adjoint operators. Therefore it might be expected that, at least where the spatial differential operator is self-adjoint, the unified transform method might yield a solution expressed as an expansion in generalised eigenfunctions. In section~\ref{ssec:Gen.Eig} we sketch a classical argument used to justify the importance of generalised eigenfunctions for the (self-adjoint) half-line Dirichlet heat operator.

In section~\ref{ssec:Aug.Eig:Defn}, we use a similar argument to show that the concept of generalised eigenfunctions is inadequate for our purposes and propose a definition for a new class of spectral objects, the augmented eigenfunctions. Section~\ref{ssec:Aug.Eig:Discussion} provides a detailed discussion of the new definition, together with a list of known diagonalisation results for spatial differential operators similar to~\eqref{eqn:S.FI} and~\eqref{eqn:S.HL}. Section~\ref{ssec:AugEig.Applied.to.examples} contains proofs of the spectral results for the operators~\eqref{eqn:S.FI} and~\eqref{eqn:S.HL}.

\subsection{Generalised eigenfunctions} \label{ssec:Gen.Eig}

Suppose we seek traditional eigenfunctions of the spatial differential operator $S$ associated with the Dirichlet problem for the heat equation and given by
\begin{align} \label{eqn:S.HeatDHL}
S &: \{f\in \mathcal{S}[0,\infty):f(0)=0\}\to \mathcal{S}[0,\infty) & Sf=-\frac{\d^2 f}{\d x^2}
\end{align}
Then $-f''(x) = \lambda^2f(x)$ implies $f(x) = Ae^{i\lambda x} + Be^{-i\lambda x}$ for some constants $A$, $B$ and the boundary condition yields $B=-A$ hence
\BE
f(x) = A'\sin(\lambda x).
\EE
But, for $f\in \mathcal{S}[0,\infty)$, we must have $A'=0$ so there are no nonzero eigenfunctions of $S$.

Instead, following Gel'fand and coauthors~\cite{GS1967a,GV1964a}, one must consider eigen\emph{functionals} or \emph{generalised eigenfunctions}, defined as functionals $F_\lambda\in(\mathcal{S}[0,\infty))'$ for which there exists some corresponding eigenvalue $\lambda^2\in\C$ with
\BE \label{eqn:GenEig.DHeat}
F_\lambda(Sf) = \lambda^2F_\lambda(f), \qquad \hsforall f \in \mathcal{S}[0,\infty).
\EE
Searching for generalised eigenfunctions of the half-line Dirichlet heat operator, we find that the definition equation~\eqref{eqn:GenEig.DHeat} holds when $\lambda$ is any real number and
\BE \label{eqn:Flambda.FourierSine}
F_\lambda(f) = \frac{1}{2\pi}\int_0^\infty\sin(\lambda x)f(x)\d x.
\EE
It is immediate that the generalised eigenfunctions we identify are complete in the space of initial data in the sense that if $F_\lambda(f)=0$ for all $\lambda\in\R$ then $f=0$. Hence it is reasonable to expect that one may express the solution of the Dirichlet problem for the heat equation as an expansion in these spectral objects. But that is the result of the classical Fourier transform approach. Indeed, the map $f(x)\mapsto F(\lambda)$ given by equation~\eqref{eqn:Flambda.FourierSine} is precisely the Fourier sine transform.

Moreover, it is a result of~\cite{GV1964a} that for \emph{any} self-adjoint operator, the generalised eigenfunctions form a complete system, so it should be possible to represent the solution of any IBVP as an expansion in the generalised eigenfunctions of the associated spatial differential operator, provided that operator is self-adjoint. However, not all differential operators are self-adjoint. Indeed, within the class of well-posed IBVP, those with self-adjoint spatial operators are exceptional. In particular, both problems~1 and~2 have non-self-adjoint operators.

Of course, this observation does not necessarily preclude the possibilities of either of the specific non-self-adjoint operators~\eqref{eqn:S.FI} and~\eqref{eqn:S.HL} having a complete system of generalised eigenfunctions or that arbitrary expansions in the generalised eigenfunctions converge. However, the argument at the beginning of section~\ref{ssec:Aug.Eig:Defn} shows not only that the transforms $F_\lambda$ defined above are not generalised eigenfunctions but that no integral transform could be composed of generalised eigenfunctions of $S$.

\medskip

It is said that the generalised eigenfunctions of a self-adjoint operator provide a spectral representation, in the sense that they diagonalise the operator. That is, a differential operator can be expressed, via its generalised eigenfunctions, as a multiplication operator. Hence it is reasonable to ask the following questions:

\begin{enumerate}
  \item[1.]{Let $f(x)\mapsto F(\lambda)$ given by equation~\eqref{eqn:introTrans.1.1} be the transform used to solve the IBVP. Is $(F_\lambda)_{\lambda\in\Gamma^+\cup\Gamma^-}$ a system of generalised eigenfunctions of $S$? If not then do these functionals have any properties similar to spectral functions?}
  \item[2.]{Is the system $(F_\lambda)_{\lambda\in\Gamma^+\cup\Gamma^-}$ complete in the space of initial data and does the expansion of an arbitrary initial datum in the system $(F_\lambda)_{\lambda\in\Gamma^+\cup\Gamma^-}$ converge?}
  \item[3.]{Can it be said that the system $(F_\lambda)_{\lambda\in\Gamma^+\cup\Gamma^-}$ of functionals diagonalises $S$ in any sense?}
\end{enumerate}

In the following section~\ref{ssec:Aug.Eig:Defn} we will discuss these questions.

\subsection{Augmented eigenfunctions: motivation and definition} \label{ssec:Aug.Eig:Defn}

Consider now the spatial differential operator~\eqref{eqn:S.HL} associated with problem~2. Let us seek generalised eigenfunctions $E_\lambda$ of $S$ that represent integral transforms, that is generalised eigenfunctions
\BE \label{eqn:AugEig.Need.1}
E_\lambda (Sf) = \lambda^3E_\lambda (f)
\EE
of the form
\BE \label{eqn:AugEig.Need.2}
E_\lambda (f) = \int_0^\infty \phi(x,\lambda) f(x) \d x,
\EE
for some integral kernel $\phi(x,\lambda)$.

Integrating by parts on the left hand side of equation~\eqref{eqn:AugEig.Need.1}, we obtain
\begin{align*}
\phi(0,\lambda) &= 0, \\
\phi'(0,\lambda) &= 0, \\
\phi'''(x,\lambda) &= (-i\lambda)^3\phi(x,\lambda).
\end{align*}
Hence
\BES
\phi(x,\lambda) = A(\lambda)\left[ e^{-i\lambda x} + \alpha^2 e^{-i\alpha\lambda x} + \alpha e^{-i\alpha^2\lambda x} \right],
\EES
for some $A$. Note that by stressing the $\lambda$ dependence of $A$ we do not intend to suggest that $A$ is expected to be an analytic function of $\lambda$, merely that this coefficient may vary with $\lambda$. But if $A(\lambda)$ is nonzero, then the integral~\eqref{eqn:AugEig.Need.2} defining $E_\lambda$ will diverge for general $f\in\mathcal{E}[0,\infty)$. Hence $S$ has no nonzero generalised eigenfunctions of this type.

The first part of question~1 has already been answered in the negative. In the following definition, we provide an answer to the more open-ended second part of question~1.

\begin{defn} \label{defn:Aug.Eig}
Let $I$ be an open real interval and let $C$ be a space of functions defined on the closure of $I$ with sufficient smoothness and decay conditions. Let $\Phi\subseteq C$ and let $L:\Phi\to C$ be a linear operator. Let $\gamma$ be an oriented contour in $\C$ and let $E=\{E_\lambda:\lambda\in\gamma\}$ be a family of functionals $E_\lambda\in C'$. Suppose there exist corresponding \emph{remainder} functionals $R_\lambda\in C'$ and \emph{eigenvalues} $z(\lambda)$ such that
\BE \label{eqn:defnAugEig.AugEig}
E_\lambda(L\phi) = z(\lambda) E_\lambda(\phi) + R_\lambda(\phi), \qquad\hsforall\phi\in\Phi, \hsforall \lambda\in\gamma.
\EE

If
\BE \label{eqn:defnAugEig.Control1}
\int_\gamma e^{i\lambda x} R_\lambda(\phi)\d\lambda = 0, \qquad \hsforall \phi\in\Phi, \hsforall x\in I,
\EE
then we say $E$ is a family of \emph{type~\textup{I} augmented eigenfunctions} of $L$ up to integration along $\gamma$.

If
\BE \label{eqn:defnAugEig.Control2}
\int_\gamma \frac{e^{i\lambda x}}{z(\lambda)}  R_\lambda(\phi)\d\lambda = 0, \qquad \hsforall \phi \in\Phi, \hsforall x\in I, 
\EE
then we say $E$ is a family of \emph{type \textup{II}~augmented eigenfunctions} of $L$ up to integration along $\gamma$.
\end{defn}

In claiming this definition provides an answer to question 1, we are not only claiming theorem~\ref{thm:Flambda.are.AugEig} (our $F_\lambda$ are augmented eigenfunctions of the relevant spatial differential operators) but also that augmented eigenfunctions can be called `spectral objects' in some reasonable sense; see section~\ref{ssec:Aug.Eig:Discussion}.

The answer to question~2 is positive. Indeed, evaluating the result of theorem~\ref{thm:TransPairSolvesIBVP} at $t=0$ immediately yields the desired completeness and convergence results. Question~3 is answered in section~\ref{ssec:Aug.Eig:Discussion} with three senses of diagonalisation, in equations~\eqref{eqn:AugEig.DiagI},~\eqref{eqn:AugEig.Just.2-2} and~\eqref{eqn:Spect.Rep.both}.

\subsection{Augmented Eigenfunctions: discourse on the definition} \label{ssec:Aug.Eig:Discussion}
In the earlier discussion of the generalised eigenfunctions~\eqref{eqn:Flambda.FourierSine} of the half-line Dirichlet heat operator~\eqref{eqn:S.HeatDHL}, we noted that these spectral objects represent the Fourier sine transform. This is the transform tailor-made to solve the initial-boundary value problem, because
\begin{enumerate}
  \item[(i)]{the transform is \emph{defined} on the space of admissible initial data}
  \item[(ii)]{the transformed initial datum (often called the spectral data) has a particularly \emph{simple time evolution}}
  \item[(iii)]{the transform is \emph{invertible} on the relevant space so the solution at a later time can be reconstructed from the time-evolved spectral data}
\end{enumerate}
At least for the purposes of solving initial-boundary value problems, it is these properties of generalised eigenfunctions that are essential.

Property~(ii) follows from the equation~\eqref{eqn:GenEig.DHeat}, usually taken to be the definition of a generalised eigenfunction. Indeed, taking spatial Fourier sine transforms of the heat equation,
\BE
q_t(x,t) = \left(\frac{\partial^2q}{\partial x^2}\right)(x,t)
\EE
one obtains
\BE
\hat{q}_t(\lambda,t) = \widehat{\left(\frac{\partial^2q}{\partial x^2}\right)}(\lambda,t).
\EE
Now equation~\eqref{eqn:GenEig.DHeat} allows us to rewrite the right hand side as $-\lambda^2\hat{q}(\lambda,t)$ which specifies a simple ordinary differential equation for the time evolution of the spectral data. The solution is
\BE
\hat{q}(\lambda,t) = e^{-\lambda^2t}\hat{q}(\lambda,0),
\EE
which is property~(ii). Hence we consider property~(ii) as a consequence of the deeper property
\begin{enumerate}
  \item[(ii$'$)]{the transform \emph{maps} the spatial differential operator \emph{to a multiplication operator} on the transformed (spectral) data}
\end{enumerate}
In solving an initial-boundary value problem, properties~(i) and~(ii$'$) would be of little use without a method for reconstructing the solution from the time-evolved spectral data, that is property~(iii). \emph{The central idea of augmented eigenfunctions is that a suitable strengthening of property~\textup{(iii)} permits a weakening of property~\textup{(ii$'$)} to the point of tautology.}
Modeled on these observations, we now discuss the definition of augmented eigenfunctions of each type.

Suppose $E=\{E_\lambda:\lambda\in\gamma\}$ is a family of type~\textup{I} augmented eigenfunctions of $L$ up to integration along $\gamma$. Then equation~\eqref{eqn:defnAugEig.AugEig} implies
\BE \label{eqn:AugEig.Just.1-1}
\int_\gamma e^{i\lambda x}E_\lambda(L\phi)\d\lambda = \int_\gamma e^{i\lambda x}z(\lambda) E_\lambda(\phi)\d\lambda + \int_\gamma e^{i\lambda x} R_\lambda(\phi)\d\lambda.
\EE
But equation~\eqref{eqn:defnAugEig.Control1} guarantees that the second integral on the right hand side of equation~\eqref{eqn:AugEig.Just.1-1} evaluates to zero. Hence
\BE \label{eqn:AugEig.DiagI}
\int_\gamma e^{i\lambda x}E_\lambda(L\phi)\d\lambda = \int_\gamma e^{i\lambda x}z(\lambda) E_\lambda(\phi)\d\lambda
\EE
and it appears formally that the type~I augmented eigenfunctions describe an integral transform which diagonalises $L$ and $\int_\gamma e^{i\lambda x}\cdot\d\lambda$ represents the corresponding inverse transform.

In the definition of type~I augmented eigenfunctions, property~(ii$'$) has been reduced merely to the definition~\eqref{eqn:defnAugEig.AugEig} of the remainder functionals, whereas property~(iii) has been strengthened; not only is $\int_\gamma e^{i\lambda x}\cdot\d\lambda$ an inverse transform but, according to equation~\eqref{eqn:defnAugEig.Control1}, when it is applied to the remainder functionals (here considered as a function of $\lambda\in\C$, with $\phi$ a parameter) it must return $0$.

There are some additional conditions required in order to make this rigourous, not least the convergence of the remaining two integrals in equation~\eqref{eqn:AugEig.Just.1-1} and an argument that this is a genuine transform pair, but this is no different from the situation for generalised eigenfunctions. Indeed, the guarantee that generalised eigenfunctions can diagonalise an operator~\cite{GV1964a} only holds for a certain class of operators: self-adjoint operators. However, a theorem of~\cite{FS2013a} combined with results in~\cite{Smi2011a,Smi2012a} shows that
\begin{itemize}
  \item[(a)]{Each operator $S$ of the form
    \begin{subequations} \label{eqn:S.GeneralFI}
    \begin{align} \label{eqn:S.GeneralFI.a}
    S &: \{f\in C^\infty[0,1]:n\mbox{ linearly independent boundary conditions}\}\to C^\infty[0,1] \\ \label{eqn:S.GeneralFI.b}
    Sf &= \left(\frac{-i\d}{\d x}\right)^nf,
    \end{align}
    \end{subequations}
    for which $n$ is even and an associated IBVP is well-posed, belongs to the class of operators which can be diagonalised by a complete family of their type~I augmented eigenfunctions.}
\end{itemize}

Now suppose instead that $E=\{E_\lambda:\lambda\in\gamma\}$ is a family of type~\textup{II} augmented eigenfunctions of $L$ up to integration along $\gamma$. Then, by equation~\eqref{eqn:defnAugEig.AugEig},
\BE \label{eqn:AugEig.Just.2-1}
\int_\gamma e^{i\lambda x}\frac{1}{z(\lambda)}E_\lambda(L\phi)\d\lambda = \int_\gamma e^{i\lambda x} E_\lambda(\phi)\d\lambda + \int_\gamma e^{i\lambda x}\frac{1}{z(\lambda)} R_\lambda(\phi)\d\lambda
\EE
and equation~\eqref{eqn:defnAugEig.Control2} ensures the final integral vanishes;
\BE \label{eqn:AugEig.Just.2-2}
\int_\gamma e^{i\lambda x}\frac{1}{z(\lambda)}E_\lambda(L\phi)\d\lambda = \int_\gamma e^{i\lambda x} E_\lambda(\phi)\d\lambda.
\EE
So formally the type~II augmented eigenfunctions specify an integral transform and $\int_\gamma e^{i\lambda x}\cdot\d\lambda$ provides an inverse transform, as in the case of type~I augmented eigenfunctions. But this time the transform pair does not diagonalise $S$ in the expected way. Instead, equation~\eqref{eqn:AugEig.Just.2-2} \emph{represents $S$ with diagonalised inverse}.

Type~\textup{II} augmented eigenfunctions have the same weakening~\eqref{eqn:defnAugEig.AugEig} of property~(ii$'$) as type~\textup{I} augmented eigenfunctions. However the strengthening of property~(iii), equation~\eqref{eqn:defnAugEig.Control2}, now insists that the inverse transform of the \emph{ratio} of the remainder functional to the eigenvalue evaluates to $0$. It is clear how the form of diagonalisation~\eqref{eqn:AugEig.Just.2-2} arises from this.

Again, there is much in this formal discussion to be justified rigourously. In particular, in investigating the convergence of the relevant integrals it becomes clear why this division by the eigenvalue is necessary. Indeed, for the operators we study here, certain families of type~II augmented eigenfunctions fail to be type~I augmented eigenfunctions precisely because the convergence of the integral~\eqref{eqn:defnAugEig.Control1} fails. Nevertheless, the rigourous arguments may be constructed for several classes of operators, listed below.

\begin{thm} \label{thm:Diag.Egs}
The operators defined by~\eqref{eqn:S.FI} or~\eqref{eqn:S.HL} can be diagonalised in the sense of equation~\eqref{eqn:AugEig.Just.2-2} by complete systems of their respective type~II augmented eigenfunctions.
\end{thm}

The proof of this theorem appears in section~\ref{ssec:AugEig.Applied.to.examples}. This result is a special case of the more general results in~\cite{FS2013a,PS2014a} summarised below:
\begin{itemize}
  \item[(b)]{If $S$ is the operator~\eqref{eqn:S.GeneralFI} for any $n\geq2$, the boundary conditions include $f(0)=0=f(1)$ and an associated IBVP is well-posed, then $S$ can be diagonalised in the sense of equation~\eqref{eqn:AugEig.Just.2-2} by a complete system of its type~II augmented eigenfunctions}
  \item[(c)]{If $S$ is the operator
	  \begin{subequations} \label{eqn:S.GeneralHL}
    \begin{align} \label{eqn:S.GeneralHL.a}
    S &: \{f\in \mathcal{E}[0,\infty):N\mbox{ linearly independent boundary conditions}\}\to \mathcal{S}[0,\infty) \\ \label{eqn:S.GeneralHL.b}
    Sf &= \left(\frac{-i\d}{\d x}\right)^nf,
    \end{align}
    \end{subequations}
    for any $n\geq2$ and an associated IBVP is well-posed, then $S$ can be diagonalised in the sense of equation~\eqref{eqn:AugEig.Just.2-2} by a complete system of its type~II augmented eigenfunctions}
\end{itemize}

It is not always possible to find a complete system of either type~I or type~II augmented eigenfunctions. However, for a wide class of operators $L$ it is possible to find a complete system $\{E_\lambda:\lambda\in\gamma^{(\mathrm{I})}\cup\gamma^{(\mathrm{II})}\}$ of functionals where some, $\{E_\lambda:\lambda\in\gamma^{(\mathrm{I})}\}$, are type~{I} augmented eigenfunctions and the rest, $\{E_\lambda:\lambda\in\gamma^{(\mathrm{II})}\}$, are type~II augmented eigenfunctions. In such a situation, $L$ is diagonalised in the sense that
\begin{subequations} \label{eqn:Spect.Rep.both}
\begin{align} \label{eqn:Spect.Rep.both.1}
\int_{\gamma^{(\mathrm{I})}} e^{i\lambda x} E_\lambda L\phi \d\lambda &= \int_{\gamma^{(\mathrm{I})}} e^{i\lambda x} z(\lambda) E_\lambda \phi \d\lambda & \hsforall \phi &\in\Phi, \hsforall x\in I, \\ \label{eqn:Spect.Rep.both.2}
\int_{\gamma^{(\mathrm{II})}} \frac{1}{z(\lambda)} e^{i\lambda x} E_\lambda L\phi \d\lambda &= \int_{\gamma^{(\mathrm{II})}} e^{i\lambda x} E_\lambda \phi \d\lambda & \hsforall \phi &\in\Phi, \hsforall x\in I.
\end{align}
\end{subequations}
This includes the result for the general two-point differential operator with monomial symbol:
\begin{itemize}
  \item[(d)]{If $S$ is the operator~\eqref{eqn:S.GeneralFI} for any $n\geq2$ with a well-posed associated initial-boundary value problem, then $S$ can be diagonalised in the sense of~\eqref{eqn:Spect.Rep.both} by a complete system of type~I and type~II augmented eigenfunctions.}
\end{itemize}

\subsection{Augmented eigenfunctions: application to problems~1 and~2} \label{ssec:AugEig.Applied.to.examples}

\begin{thm} \label{thm:Flambda.are.AugEig}
\begin{enumerate}
  \item[(i)]{Suppose $S$ is defined by equation~\eqref{eqn:S.FI}, $\Gamma^\pm$ are the contours shown on figure~\ref{fig:FI-cont} and the functionals $F_\lambda$ are defined by equations~\eqref{eqn:introTrans.1.1a},~\eqref{eqn:introTrans.1.1c},~\eqref{eqn:introTrans.1.1d}. Then $\{F^+_\lambda:\lambda\in\Gamma^+\}\cup\{F^-_\lambda:\lambda\in\Gamma^-\}$ are type~II augmented eigenfunctions of $S$ up to integration along $\Gamma^+\cup\Gamma^-$, with eigenvalues $\lambda^3$.
		}
  \item[(ii)]{Suppose $S$ is defined by equation~\eqref{eqn:S.HL}, $\Gamma^\pm$ are the contours shown on figure~\ref{fig:HL-cont} and the functionals $F_\lambda$ are defined by equations~\eqref{eqn:introTrans.1.1a},~\eqref{eqn:introTrans.1.1e},~\eqref{eqn:introTrans.1.1f}. Then $\{F^+_\lambda:\lambda\in\Gamma^+\}\cup\{F^-_\lambda:\lambda\in\Gamma^-\}$ are type~II augmented eigenfunctions of $S$ up to integration along $\Gamma^+\cup\Gamma^-$, with eigenvalues $\lambda^3$.
		}
\end{enumerate}
\end{thm}

\begin{proof}~\\
\noindent(i)\hspace{1em}Integration by parts yields
\BE
F_\lambda^\pm(Sf) = \lambda^3 F_\lambda^\pm(f) + R_\lambda^\pm(f), \qquad \hsforall \lambda\in\Gamma^\pm, \hsforall f\in\mathcal{D}(S),
\EE
where
\begin{align}
R_\lambda^+(f) &= -\frac{i}{2\pi}f''(0) - \frac{\lambda}{2\pi}f'(0), \\
R_\lambda^-(f) &= -e^{-i\lambda}\frac{i}{2\pi}f''(1).
\end{align}
That is, equation~\eqref{eqn:defnAugEig.AugEig} holds with eigenvalues $\lambda^3$ and remainder functionals entire in $\lambda$. Moreover, $R_\lambda^\pm/\lambda^3$ are meromorphic functions of $\lambda$, each with no pole in the sectors lying to the left of $\Gamma^\pm$. As $R_\lambda^+$ is decaying like $O(\lambda^{-2})$ as $\lambda\to\infty$, Jordan's lemma yields
\BE
\int_{\Gamma^+} e^{i\lambda x}\frac{R_\lambda^+(f)}{\lambda^3}\d\lambda = 0.
\EE
Similarly,
\BE
\int_{\Gamma^-} e^{i\lambda x}\frac{R_\lambda^-(f)}{\lambda^3}\d\lambda = \int_{\Gamma^-} e^{-i\lambda(1-x)}\left[-\frac{i}{2\pi\lambda^3}f''(1)\right]\d\lambda = 0.
\EE

\noindent(ii)\hspace{1em}Integration by parts yields
\BE
F_\lambda^\pm(Sf) = \lambda^3 F_\lambda^\pm(f) + R_\lambda^\pm(f), \qquad \hsforall \lambda\in\Gamma^\pm, \hsforall f\in\mathcal{D}(S),
\EE
where
\begin{align}
R_\lambda^+(f) &= \frac{i}{2\pi}f''(0) - \frac{\lambda}{2\pi}f'(0), \\
R_\lambda^-(f) &= -\frac{i}{2\pi}f''(0) + \frac{\lambda}{2\pi}f'(0).
\end{align}
That is, equation~\eqref{eqn:defnAugEig.AugEig} holds with eigenvalues $\lambda^3$ and remainder functionals entire in $\lambda$. Moreover, $R_\lambda^\pm/\lambda^3$ are meromorphic functions, each with no pole to the left of $\Gamma^\pm$, and $O(\lambda^{-2})$ decay as $\lambda\to\infty$. Hence, by Jordan's lemma,
\BE \label{eqn:proof.Flambda.are.AugEig:prob2}
\int_{\Gamma^\pm} e^{i\lambda x}\frac{R_\lambda^\pm(f)}{\lambda^3}\d\lambda = 0.
\EE
\end{proof}

The reason we work in $\mathcal{E}[0,\infty)$ for problem~2 now becomes clear. Suppose that we wish to work in the full Schwartz space $\mathcal{S}[0,\infty)$. Then our forward transform is not defined on all of $\Gamma^-$ (see equations~\eqref{eqn:introTrans.1.1a} and~\eqref{eqn:introTrans.1.1f} and figure~\ref{fig:HL-cont}) so we must replace $\Gamma^-$ with some new contour, $\gamma^-$ say, which lies wholly in the closure of the lower half-plane. But then equation~\eqref{eqn:proof.Flambda.are.AugEig:prob2} fails, as we pick up a contribution from the pole at zero.

\begin{proof}[Proof of theorem~\ref{thm:Diag.Egs}]
We give the proof for the operator~\eqref{eqn:S.FI}. The proof for the operator~\eqref{eqn:S.HL} is similar.

Theorem~\ref{thm:Flambda.are.AugEig}(i) gives the family of type~II augmented eigenfunctions. Moreover, these augmented eigenfunctions describe the forward transform used in theorem~\ref{thm:TransPairSolvesIBVP}. Evaluating that result at $t=0$, we obtain completeness of the system of functionals.

The integrals
\BE
\int_{\Gamma^\pm} e^{i\lambda x}F_\lambda^\pm f\d\lambda
\EE
converge as they are the integrals appearing in theorem~\ref{thm:TransPairSolvesIBVP}, for $t=0$. Thus the integrals in equation~\eqref{eqn:AugEig.Just.2-2} converge.
\end{proof}

\subsubsection*{Acknowledgement}
The research leading to these results has received funding from the European Union's Seventh Framework Programme (FP7-REGPOT-2009-1) under grant agreement n$^\circ$ 245749.

\end{document}